\documentclass[11pt, a4paper]{article}%
\usepackage[colorlinks]{hyperref}

\usepackage{amsmath}%
\usepackage{amsfonts}%
\usepackage{amssymb}%

\newtheorem{theorem}{Theorem}

\newtheorem{corollary}[theorem]{Corollary}

\newtheorem{definition}[theorem]{Definition}
\newtheorem{example}[theorem]{Example}

\newtheorem{lemma}[theorem]{Lemma}
\newtheorem{notation}[theorem]{Notation}

\newtheorem{proposition}[theorem]{Proposition}
\newtheorem{remark}[theorem]{Remark}

\newenvironment{proof}[1][Proof]{\noindent\textbf{#1.} }{\
\rule{0.5em}{0.5em}}

\def\noi{\noindent}

\def\C{\mathbb{C}}
\def\N{\mathbb{N}}

\def\V{\mathbb{V}}
\def\W{\mathbb{W}}
\def\ring{\C[\vec{x},\vec{y}]}

\def\vec{\mathbf}

\begin{document}

\title{A decision method for the integrability of differential-algebraic Pfaffian systems \thanks{Partially supported by the following grants: UBACYT 20020110100063 (2012-2015) / Math-AmSud SIMCA ``Implicit Systems, Modelling and Automatic Control" (2013-2014).}}

\author{Lisi D'Alfonso$^\natural$ \and Gabriela Jeronimo$^\sharp$ \and Pablo Solern\'o$^\sharp$\\[4mm]
{\small $\natural$ Departamento de Ciencias Exactas, Ciclo B\'asico Com\'un, Universidad de Buenos Aires,} \\{\small Ciudad Universitaria, 1428, Buenos Aires, Argentina}\\[2mm]
{\small $\sharp$ Departamento de Matem\'atica and IMAS, UBA-CONICET,} \\{\small Facultad de Ciencias Exactas y Naturales,Universidad de Buenos Aires,}\\ {\small Ciudad Universitaria, 1428, Buenos Aires, Argentina}\\[3mm]
{\small E-mail addresses: lisi@cbc.uba.ar, jeronimo@dm.uba.ar, psolerno@dm.uba.ar}
}
\date{}
\maketitle

\begin{abstract}
We prove an effective integrability criterion for differential-algebraic Pfaffian systems leading to a decision method of consistency with a triple exponential complexity bound. As a byproduct, we obtain an upper bound for the order of differentiations in the differential Nullstellensatz for these systems.
\end{abstract}



\section{Introduction}

Let $\vec{x}:=x_1,\ldots,x_m$ and $\vec{y}:=y_1,\ldots,y_n$ be two sets of variables; the first ones represent the independent variables (i.e. those defining the partial derivations) and the second ones are considered as differential unknowns.

The notion of \emph{Pfaffian system} is introduced by J.F. Pfaff in \cite{pfaff} and in its simplest form it is defined as a system of partial differential equations of the type:

\begin{equation}\label{Pfaff general}
\Sigma:=\left\{ \begin{matrix}
\dfrac{\partial y_i}{\partial x_j}=f_{ij}(\vec{x},\vec{y})\end{matrix} \right. \ i=1,\ldots,n,\ j=1,\ldots,m,\end{equation}
 where each $f_{ij}$ is an analytic function around a certain fixed point $(\vec{x}_0,\vec{y}_0)\in \C^m\times \C^n$.

The properties of these systems were extensively studied during the XIXth century by notable mathematicians as Jacobi \cite{jacobi}, Clebsch \cite{clebsch} and Frobenius \cite{frobenius} (see also \cite{hawkins} for a detailed historical approach). One of the main problems considered is the so-called \emph{complete integrability of a Pfaffian system $\Sigma$}: the existence of a neighborhood $U$ of the point $(\vec{x}_0,\vec{y}_0)$ such that for all point $(\widehat{\vec{x}},\widehat{\vec{y}})\in U$ there exists a solution $\gamma$ of $\Sigma$ such that $\gamma(\widehat{\vec{x}})=\widehat{\vec{y}}$. A complete solution of this problem is  known today as the ``Frobenius Theorem" (see \cite{frobenius}): \\

\noindent \textbf{Frobenius Theorem} \emph{Let $\Sigma$ be the Pfaffian system (\ref{Pfaff general}). Then $\Sigma$ is completely integrable in $(\vec{x}_0,\vec{y}_0)$ if and only if for all indices $i,j,k$ with $i=1,\ldots,n$, and $j,k=1,\ldots,m$, the function $D_j(f_{ik})-D_k(f_{ij})$ vanishes in a neighborhood of $(\vec{x}_0,\vec{y}_0)$}.\\

Here, $D_j(h)$ denotes the $j$-th total derivative with respect to $\Sigma$ of any analytic function $h$ in the variables $(\vec{x},\vec{y})$, namely $D_j(h):=\dfrac{\partial h}{\partial x_j}+\sum_i \dfrac{\partial h}{\partial y_i}\, f_{ij}$. Observe that the vanishing of the functions $D_j(f_{ik})-D_k(f_{ij})$ is clearly a necessary condition because of the equality of the mixed derivatives of an analytic function, but the sufficiency is not obvious. \\

We remark that Frobenius Theorem is more general than the statement above because it remains true also for systems of differential linear 1-forms (today called \emph{Pfaffian forms}).
In this sense Frobenius's article may be considered as a main source of inspiration for the transcendental work by Cartan about integrability of exterior differential systems developed in the first half of the XXth century (see \cite{cartan}).

In Cartan's theory two notions play a main r\^ole: the \emph{prolongation} and the \emph{involutivity}  of an exterior differential system. In the particular case of differential equation systems these notions can be easily paraphrased. The prolongation of a system consists simply in applying derivations to it. 
On the other hand, an involutive system is a system with no hidden integrability conditions; in other words, no prolongation is necessary to find new constraints. Roughly speaking, the involutive systems are those to which the classical method of resolution by means of a power series with indeterminate coefficients (known as Frobenius method) can be applied in order to search for a solution or to decide that the system is not integrable.

A main general result related to the integrability in Cartan's theory is the  Cartan-Kuranishi Principle: \emph{Any (generic) exterior differential system can be reduced  to an equivalent involutive system  by a finite number of prolongations and projections.} This result, conjectured by Cartan, is finally   proved by Kuranishi in \cite{kuranishi} (see also \cite{matsuda}). Despite the algebraic precisions given during the '60s and '70s (see for instance \cite{spencer,guillemin,singer}) and more recently in \cite{pommaret} and \cite{seiler}, up to our knowledge, no effective \emph{a priori} upper bounds for the number of required differentiations nor projections are established.\\

In this paper we use arguments close to the Cartan-Kuranishi Principle (namely, prolongations and projections) in order to study the integrability (not necessarily complete) of differential-algebraic Pfaffian systems. More precisely, \emph{a differential-algebraic Pfaffian system} is a system of partial differential equations $\Sigma$ as follows:
\begin{equation} \label{Pfaff algebraic}
\Sigma:=\left\{ \begin{matrix}
\dfrac{\partial \vec{y}}{\partial \vec{x}}=\vec{f}(\vec{x},\vec{y})\cr
\vec{g}(\vec{x},\vec{y})=0\end{matrix} \right.
\end{equation}
where $\vec{f}:=(f_{ij})_{ij}$ and $\vec{g}:=g_1,\ldots ,g_s$ are polynomials in $\C[\vec{x},\vec{y}]$. As before,   $\vec{x}:=(x_1,\ldots,x_m)$ and $\vec{y}:=(y_1,\ldots,y_n)$ denote sets of variables.

By means of prolongations and projections we construct a decreasing chain of algebraic varieties $\C^{m+n}\supseteq\V_0\supseteq \V_1\supseteq\cdots $ which becomes stationary at most at the $(m+n+1)$th step. If we denote by $\V_\infty$ the smallest variety of the chain, we prove the following criterion for integrability (see Theorem \ref{criterio} and the comments that follow it):

 \begin{theorem} \label{thm_non_aut} The system $\Sigma$ is integrable if and only if the algebraic variety $\V_\infty$ is non-empty. The analytic variety of all the regular points of $\V_\infty$ is the biggest analytic variety containing all the graphs of analytic solutions of $\Sigma$.
 \end{theorem}

Moreover, by means of basic algorithms from commutative algebra, the criterion can be transformed in a decision algorithm which runs within a complexity of order $(n m \sigma d)^{2^{O(n+m)^3}}$, where $d$ is an upper bound for the degrees of the involved polynomials and $\sigma:= \max\{1, s\}$ (see Theorem \ref{thm:decision}).
As a byproduct, we obtain an upper bound of the same order for the order of differentiations in the differential Nullstellensatz for differential-algebraic Pfaffian systems (see Theorem \ref{diffnulls}). In this sense, the present work can be seen as a continuation of \cite{DJS14}.\\

The paper is organized as follows. In Section \ref{sec:criterion}, first, we introduce the notation we use throughout the paper; then, we show how to reduce the integrability problem of general (non-autonomous) differential Pfaffian systems to the autonomous case; finally, we prove our integrability criteria. Section \ref{sec:quantitative} is devoted to analyzing quantitative aspects of the problem: we present an effective decision method for the integrability of differential-algebraic Pfaffian systems and we prove an upper bound for the order in the differential Nullstellensatz for these systems.

\section{A geometrical criterion for the integrability of differential-algebraic Pfaffian systems}\label{sec:criterion}

In this section we exhibit a necessary and sufficient criterion for the integrability of differential-algebraic Pfaffian systems in terms of the dimension of a finite decreasing sequence of algebraic varieties associated to the differential system (see Definition \ref{chainideals} and Theorem \ref{criterio}). These varieties can be constructed explicitly by means of three basic operations: prolongations (i.e. differentiations), linear projections and reductions (i.e. computation of radicals of polynomial ideals).

\subsection{Notations} \label{notation}

Let $m,n\in \mathbb{N}$. We consider two sets of variables $\vec{x}:=x_1,\ldots,x_m$ (the so-called \emph{independent variables}) and $\vec{y}:=y_1,\ldots,y_n$ (the variables  playing the role of \emph{differential unknowns}).
For each pair $(i,j)$, $1\le i\le n$, $1\le j\le m$, let $f_{ij}$ be a polynomial in the ring $\C[\vec{x},\vec{y}]$. Finally, let $\vec{g}:=g_1,\ldots,g_s\in \C[\vec{x},\vec{y}]$ be another finite set of polynomials.

A \emph{differential-algebraic Pfaffian system} is defined as a partial differential system $\Sigma$ of the type:

\begin{equation} \label{general_completo} \Sigma = \left\{ \begin{matrix}
\ \dfrac{\partial y_i}{\partial x_j}= f_{ij}(\vec{x},\vec{y}),& \ 1\le i\le n,\ 1\le j\le m,\cr
g_k(\vec{x}, \vec{y})= 0&\ 1\le k\le s,
\end{matrix} \right. \qquad
\end{equation}
or in its simplified form:
\[
\Sigma = \left\{
\begin{matrix}
\ \dfrac{\partial \vec{y}}{\partial \vec{x}}= \vec{f}(\vec{x},\vec{y}),\cr
\vec{g}(\vec{x},\vec{y})=0\
\end{matrix} \right.
\]
where $\vec{f}$ denotes the set of polynomials $f_{ij}$.
If the polynomials $\vec{f}$ and $\vec{g}$ do not depend on the variables $\vec{x}$, we say that the Pfaffian system $\Sigma$ is   \emph{autonomous}.

For each index $j=1,\ldots,m$ and any polynomial $h\in\C[\vec{x},\vec{y}]$ we define \emph{the total derivative with respect to $x_j$ induced by $\Sigma$} as
\[
D_j(h):=\dfrac{\partial h}{\partial x_j}+\sum_{i=1}^n \dfrac{\partial h}{\partial y_i}f_{ij}.
\]
Observe that $D_j(h)$ belongs to the polynomial ring $\C[\vec{x},\vec{y}]$ and that the operator $D_j$ is a derivation in this ring.

\subsection{Frobenius compatibility conditions}

We are interested in the \emph{integrability} (or \emph{solvability}) of differential-algebraic Pfaffian systems. This notion should be understood as the existence of an \emph{analytic} solution $\gamma$ defined in an open subset of $\C^n$ with target space $\C^m$.

A stronger notion is that of \emph{complete integrability}, which means that there exists a neighborhood $U\subset \C^{m+n}$ around $(\vec{x}_0,\vec{y}_0)$, such that for any $(\widehat{\vec{x}},\widehat{\vec{y}})\in U$ there exists a solution $\gamma$ of $\Sigma$ verifying $\gamma(\widehat{\vec{x}})=\widehat{\vec{y}}$ (note that this notion only makes sense in the case that no algebraic constraints appear, because such a constraint never contains a nonempty open set).
If this is the case for the system $\Sigma$, the classical Frobenius Theorem (see for instance \cite[Ch.X, \S9]{dieudonne}) gives a simple criterion for complete integrability: \emph{$\Sigma$ is completely integrable around $(\vec{x}_0,\vec{y}_0)$ if and only if the compatibility conditions $D_j(f_{i\ell})-D_\ell(f_{ij})\equiv 0$ hold for all triplets $i,j,\ell$ in a suitable neighborhood of $(\vec{x}_0,\vec{y}_0)$.}\\

Our systems are in a certain sense more general than those considered in Frobenius's result (more constraints appear) and we are interested in a weaker notion of integrability. In order to begin our analysis, we introduce the following ideal which will enable us to deal with the Frobenius compatibility conditions and will be involved in the first step of our criterion:

\begin{notation} \label{ideal}
For a given Pfaffian system $\Sigma$ as in (\ref{general_completo}) we denote by $\mathfrak{F}$ the ideal in the polynomial ring $\ring$ generated by the polynomials $D_j(f_{ik})-D_k(f_{ij})$ for all indices $i,j,k$.
\end{notation}

If no algebraic constraints appear in $\Sigma$, the classical Frobenius Theorem in the differential-algebraic setting can be restated as follows:  \emph{$\Sigma$ is  {completely integrable} if and only if the polynomial ideal $\mathfrak{F}$ is zero.}
On the other hand, if such a system $\Sigma$ is integrable, the ideal $\mathfrak{F}$ must be properly contained in $\ring$, since if $\gamma$ is a solution, all polynomial in $\mathfrak{F}$ vanish at the vector $\gamma(\vec{x})$. However, this condition (namely, the properness of $\mathfrak{F}$) is not enough to guarantee simple integrability:\\

\noi \textbf{Example}: Consider the Pfaffian system
\[\dfrac{\partial y}{\partial x_1}=y^2\ ,\ \dfrac{\partial y}{\partial x_2}=y^2+1.\]
In this case, $n= 1$, $m=2$, and no constraint appears. It is obvious that if $\gamma$ is a solution, then it is not a constant function because the polynomials $y^2$ and $y^2+1$ have no common zeros. The ideal $\mathfrak{F}$ is generated by the polynomial $2y(y^2+1)-2yy^2=2y$ and then it is proper. On the other hand, any solution $\gamma$ must verify the equation $2\gamma=0$ and in particular $\gamma$ is a constant, leading to a  contradiction.

\subsection{From non-autonomous to autonomous systems} \label{sec:nonaut}

We begin by showing that the problem of the integrability of a (general) differential Pfaffian system
\[
\Sigma = \left\{
\begin{matrix}
\ \dfrac{\partial \vec{y}}{\partial \vec{x}}= \vec{f}(\vec{x},\vec{y})\cr
\vec{g}(\vec{x},\vec{y})=0\
\end{matrix} \right. \quad
\]
can be reduced to analyzing the integrability of an autonomous Pfaffian system.

To this end, we transform $\Sigma$ into an autonomous Pfaffian system $\Sigma_{\textrm{aut}}$ in the obvious way, by introducing new differential unknowns  $\vec{w}:=w_1,\ldots,w_m$ (as many as independent variables):
\begin{equation}\label{sigma_aut}
\Sigma_{\textrm{aut}} = \left\{
\begin{matrix}
\ \dfrac{\partial \vec{y}}{\partial \vec{x}} = \vec{f}(\vec{w},\vec{y}) \\[3mm]
\ \!\!\!\!\!\!\!\!\!\!\!\!\!\!\!\!\!\!
\ \dfrac{\partial \vec{w}}{\partial \vec{x}} =\delta \\[2mm]
\ \vec{g}(\vec{w},\vec{y})=0\
\end{matrix} \right.
\end{equation}
where $\delta$ denotes the Dirac-symbol (i.e. $\delta_{j\ell}=1$ if  $j=\ell$, and $0$ otherwise). Clearly $\Sigma_{\textrm{aut}}$ is an autonomous system in $m+n$ unknowns and $m$ independent variables; moreover, it is equivalent to $\Sigma$ from the integrability point of view:

\begin{proposition}
The Pfaffian system $\Sigma$ is integrable if and only if the autonomous Pfaffian system $\Sigma_{\textrm{aut}}$ is integrable.
\end{proposition}

\begin{proof} Let $\gamma:\mathcal U\subset \C^m\to \C^n$ be a solution of $\Sigma$; then  $\psi:\mathcal U\to \C^{m+n}$ defined as $\psi(\vec{x}):=(\vec{x},\gamma(\vec{x}))$ is a solution of $\Sigma_{\textrm{aut}}$. Reciprocally, let $\psi:\mathcal{V}\to \C^{m+n}$ be a solution of $\Sigma_{\textrm{aut}}$ in a connected neighborhood of a certain $\vec{x}_1\in\C^m$. From the second part of the defining equations of $\Sigma_{\textrm{aut}}$ one deduces that $\psi(\vec{x})=(\vec{x}+\lambda,\xi(\vec{x}))$ for a suitable vector $\lambda\in \C^m$ and an analytic function $\xi: \mathcal{V}\to \C^n$. Take $\vec{x}_0:=\vec{x}_1+\lambda$ and $\mathcal{U}:=\mathcal{V}+\lambda$. Define $\gamma(\vec{x}):=\xi(\vec{x}-\lambda)$. Thus $\gamma:\mathcal{U}\to \C^n$ verifies:
\[ \dfrac{\partial \gamma_i(\vec{x})}{\partial x_j}= \dfrac{\partial \psi_{m+i}(\vec{x}-\lambda)}{\partial x_j}=f_{ij}(\psi(\vec{x}-\lambda))=f_{ij}((\vec{x}-\lambda)+\lambda,\xi(\vec{x}-\lambda))=
f_{ij}(\vec{x},\gamma(\vec{x}))\]
and
\[
\vec{g}(\vec{x},\gamma(\vec{x}))=\vec{g}((\vec{x}-\lambda)+\lambda,
\xi(\vec{x}-\lambda))=\vec{g}(\psi(\vec{x}-\lambda))=0.
\]
In other words, $\gamma$ is a solution of $\Sigma$.
\end{proof}

\subsection{Prolongation chains}

Taking into account the result of the previous section, for simplicity, we restrict now our attention to considering the \emph{autonomous case},
namely, differential-algebraic Pfaffian systems of the form
\begin{equation}\label{autsyst}
\Sigma = \left\{
\begin{matrix}
\ \dfrac{\partial \vec{y}}{\partial \vec{x}}= \vec{f}(\vec{y}),\cr
\vec{g}(\vec{y})=0\
\end{matrix} \right.
\end{equation}
where $\vec{f}$ and $\vec{g}$ are finite sets of polynomials in $\C[\vec{y}]$.
In this setting, for a polynomial $h\in\C[\vec{y}]$,
 the total derivative of $h$  with respect to $x_j$ induced by $\Sigma$ is
\[
D_j(h)=\sum_{i=1}^n \dfrac{\partial h}{\partial y_i}f_{ij}\in \C[\vec{y}].
\]

First, we introduce the so-called \emph{prolongation} of an ideal in $\C[\vec{y}]$:

\begin{definition}\label{rulodef}
Let $I\subset\C[\vec{y}]$ be an arbitrary ideal and $\vec{F}$ be a finite system of generators of $I$. We denote by $\widetilde{I}$ the ideal of the polynomial ring $\C[\vec{y}]$ generated by $I$ and the polynomials $D_j(f)$ for all $f\in \vec{F}$ and $j=1,\ldots,m$.
\end{definition}

\begin{proposition} \label{rulo}
The ideal $\widetilde{I}$ is independent of the system of generators $\vec{F}$ of $I$.
\end{proposition}

\begin{proof}
 Let $\vec{G}$ be another system of generators of $I$. It suffices to see that for any   $k=1,\ldots ,m$ and any $g\in \vec{G}$ the polynomial $D_k(g)$ belongs to the ideal generated by $I$ and the polynomials $D_j(f)$. Writing $g=\displaystyle{\sum_{f\in \vec{F}} p_f\, f}$ and using that $D_j$ is a derivation in $\C[\vec{y}]$, we deduce:
\[ D_j(g)=\sum_{f\in \vec{F}} D_j(p_f\, f)=\sum_{f\in \vec{F}} D_j(p_f)\cdot f\ +\ \sum_{f\in \vec{F}} p_f\cdot D_j(f).\]
Since the first term belongs to $(\vec{F})=I$ and the second one to the ideal generated by the polynomials $D_j(f)$, the proposition follows.
\end{proof}\\

We start proving an elementary test to check integrability by means of a prolongation and a projection, which will be the key result to our criteria.

\begin{lemma} \label{key lemma auto}
Let $\Sigma$ be an autonomous differential-algebraic Pfaffian system:
\[
\Sigma = \left\{
\begin{matrix}
\ \dfrac{\partial \vec{y}}{\partial \vec{x}}= \vec{f}(\vec{y}),\cr
\vec{g}(\vec{y})=0\
\end{matrix} \right.
\]
where $\vec{f}$ and $\vec{g}$ are finite sets of polynomials in $\C[\vec{y}]$. Let $\mathcal{I}:=\sqrt{(\vec{g})}$ and $\mathcal{J}:=\sqrt{\mathfrak{F}+\widetilde{\mathcal{I}}}$, where  $\mathfrak{F}\subset \C[\vec{y}]$ is the ideal generated by the Frobenius conditions, as in Notation \ref{ideal}, and let  $\V\supseteq\mathbb{W}$ be the varieties defined by $\mathcal{I}$ and $\mathcal{J}$ respectively. Suppose that $\V$ and $\W$ are the same nonempty variety in a neighborhood of a common regular point $Q\in\C^n$.
Then $\Sigma$ is integrable. Moreover, there exists a solution of $\Sigma$ passing through the point $Q$ and contained in $\W$.
\end{lemma}

\begin{proof}
From the hypotheses, the ideals $\mathcal{I}$ and $\mathcal{J}$ define the same nonempty algebraic variety $C$ locally around $Q$.
Let $\vec{g}_0$ be a finite system of generators of $\mathcal{I}$. Since $Q$ is a regular point of $C$, we can apply the Theorem of Implicit Functions to the system $\vec{g}_0=0$ and so, without loss of generality (reordering the variables if necessary) we may suppose that there exists an integer $r>0$
such that, if $\overline{\vec{y}}:=y_1,\ldots,y_r$ and $\widehat{\vec{y}}:=y_{r+1},\ldots,y_n$, the variety $C$ is locally defined around $Q=(\overline{\vec{y}}_0,\widehat{\vec{y}}_0)$ as the graph of an infinite differentiable function $\varphi$ defined in a neighborhood $\mathcal{U}\subset \C^r$ of $\overline{\vec{y}}_0$ with target space $\C^{n-r}$. The integer $r$ is the dimension of $C$, the rank of the Jacobian matrix $\dfrac{\partial \vec{g}_0}{\partial \vec{y}}$ in a neighborhood of $Q$ in $C$ is constant and equal to $n-r$, and moreover, the submatrix $\dfrac{\partial \vec{g}_0}{\partial \widehat{\vec{y}}}$ has locally maximal rank $n-r$.

\noi Consider the following partial differential (not necessarily algebraic) Pfaffian system $\overline{\Sigma}$ induced by $\Sigma$ and $\varphi$:
\[
\overline{\Sigma}:=\left\{
\begin{matrix}
\ \dfrac{\partial \overline{\vec{y}}}{\partial \vec{x}}= \overline{\vec{f}}(\vec{\overline{\vec{y}},\varphi(\overline{\vec{y}})})
\end{matrix} \right. ,
\]
where $\overline{\vec{f}}:=(f_{ij})$ with $i=1,\ldots,r$, $j=1,\ldots,m$.

\noi The proof of the lemma will be achieved by showing the  following facts concerning the system $\overline{\Sigma}$:

\begin{itemize}
\item{\textbf{Claim 1.}} \emph{The system $\overline{\Sigma}$ is completely integrable in a neighborhood of $\overline{\vec{y}_0}$.}

\item{\textbf{Claim 2.}} \emph{If $\mu$ is a solution of $\overline{\Sigma}$ around $\overline{\vec{y}_0}$, then $\gamma:=(\mu,\varphi\circ \mu)$ is a solution of $\Sigma$.}
\end{itemize}

\noi \emph{Proof of Claim 1.} From Frobenius Theorem it suffices to prove that for each pair $j,k$, $1\le j,k\le m$, and all $i=1,\ldots,r$, the relation
\[
\sum_{h=1}^r \left(\dfrac{\partial {f}_{ij}(\overline{\vec{y}}, \varphi(\overline{\vec{y}}))}{\partial y_h}\ {f}_{hk}-\dfrac{\partial {f}_{ik}(\overline{\vec{y}}, \varphi(\overline{\vec{y}}))}{\partial y_h}\ {f}_{hj}\right)\bigg|_{(\overline{\vec{y}}, \varphi(\overline{\vec{y}}))}=0
\]
holds for all $\overline{\vec{y}}\in \mathcal{U}$ (where $\mathcal{U}$ a suitable neighborhood of $\overline{\vec{y}}_0$).

\noi From the Chain Rule, for any $q=1,\ldots,m$, and $i,h=1,\ldots,r$,  we have the equality
\[\dfrac{\partial {f}_{iq}(\overline{\vec{y}}, \varphi(\overline{\vec{y}}))}{\partial y_h}=\dfrac{\partial {f}_{iq}}{\partial y_h}\bigg|_{(\overline{\vec{y}}, \varphi(\overline{\vec{y}}))}+\sum_{\ell=r+1}^n  \dfrac{\partial {f}_{iq}}{\partial y_\ell}\bigg|_{(\overline{\vec{y}}, \varphi(\overline{\vec{y}}))} \dfrac{\partial \varphi_\ell}{\partial y_h}(\overline{\vec{y}})
\ ,\]
where $\varphi_{r+1},\ldots,\varphi_n$ are the coordinates of the function $\varphi$.

\noi These equalities can be collected in a unique relation in terms of a product of matrices evaluated in $(\overline{\vec{y}}, \varphi(\overline{\vec{y}}))$ which describes the Frobenius compatibility conditions for the system $\overline{\Sigma}$ as follows:

\begin{equation} \label{Frobenius bar}
\left(\dfrac{\partial {f}_{ij}}{\partial \overline{\vec{y}}} + \dfrac{\partial {f}_{ij}}{\partial \widehat{\vec{y}}}\cdot \dfrac{\partial\varphi}{\partial \overline{\vec{y}}}\right)\cdot\left(
               \begin{array}{c}
                 f_{1k} \\
                 \vdots \\
                 f_{rk} \\
               \end{array}
             \right)= \left(\dfrac{\partial {f}_{ik}}{\partial \overline{\vec{y}}} + \dfrac{\partial {f}_{ik}}{\partial \widehat{\vec{y}}}\cdot \dfrac{\partial\varphi}{\partial \overline{\vec{y}}}\right)\cdot\left(
               \begin{array}{c}
                 f_{1j} \\
                 \vdots \\
                 f_{rj} \\
               \end{array}
             \right)
\end{equation}
where  $\dfrac{\partial\varphi}{\partial \overline{\vec{y}}}$ denotes the $(n-r)\times r$ Jacobian matrix of the map $\varphi$ with respect to the variables $\overline{\vec{y}}$.

\noi On the other hand, differentiating the identity $\vec{g}_0(\overline{\vec{y}},\varphi(\overline{\vec{y}}))\equiv 0$ and multiplying by the column vector $(f_{ij})_{1\le i \le r}$, we infer that, for all $j=1,\dots, m$,

\begin{equation} \label{equality1}
\dfrac{\partial \vec{g}_0}{\partial \overline{\vec{y}}}\cdot \left(
               \begin{array}{c}
                 f_{1j} \\
                 \vdots \\
                 f_{rj} \\
               \end{array}
             \right)\ +\ \dfrac{\partial \vec{g}_0}{\partial \widehat{\vec{y}}}\cdot \dfrac{\partial\varphi}{\partial \overline{\vec{y}}}\cdot \left(\begin{array}{c}
                 f_{1j} \\
                 \vdots \\
                 f_{rj} \\
               \end{array}\right)=0
\end{equation}
holds at all the points $(\overline{\vec{y}},\varphi(\overline{\vec{y}}))$.

\noi Finally, from the hypothesis $C$ coincides with $\V$ and $\W$ in a neighborhood of the point $Q$, we may suppose (shrinking the neighborhood $\mathcal{U}$ of $\overline{\vec{y}}_0$ if necessary) that $h(\overline{\vec{y}},\varphi(\overline{\vec{y}}))=0$ for all $h\in \mathcal{J}$, in particular since $D_j(g)\in \mathcal{J}$ for all $j=1,\ldots,m$ and all $g\in \vec{g}_0$ we have
\[
\dfrac{\partial \vec{g}_0}{\partial \vec{y}}\cdot \left(
                                               \begin{array}{c}
                                                 f_{1j} \\
                                                 \vdots \\
                                                 f_{nj} \\
                                               \end{array}
                                             \right)\ =   \ 0
\]
in any point $(\overline{\vec{y}},\varphi(\overline{\vec{y}}))$, for all index $j$. Again, separating the variables $\vec{y}$ as before, this equality gives
\begin{equation} \label{equality2}
\dfrac{\partial \vec{g}_0}{\partial \overline{\vec{y}}}\cdot \left(
                                               \begin{array}{c}
                                                 f_{1j} \\
                                                 \vdots \\
                                                 f_{rj} \\
                                               \end{array}
                                             \right)\ +
\dfrac{\partial \vec{g}_0}{\partial \widehat{\vec{y}}}\cdot \left(
                                               \begin{array}{c}
                                                 f_{(r+1)j} \\
                                                 \vdots \\
                                                 f_{nj} \\
                                               \end{array}
                                             \right)\ =   \ 0.
\end{equation}

\noi Comparing equalities (\ref{equality1}) and (\ref{equality2}) and recalling that the matrix $\dfrac{\partial \vec{g}_0}{\partial \widehat{\vec{y}}}$ has full rank $n-r$ we infer that
\begin{equation} \label{despeje}
\dfrac{\partial\varphi}{\partial \overline{\vec{y}}}\cdot \left(\begin{array}{c}
                 f_{1j} \\
                 \vdots \\
                 f_{rj} \\
               \end{array}\right)=\left(
                                               \begin{array}{c}
                                                 f_{(r+1)j} \\
                                                 \vdots \\
                                                 f_{nj} \\
                                               \end{array}
                                             \right)
\end{equation}
in any point $(\overline{\vec{y}},\varphi(\overline{\vec{y}}))$ and for all index $j$.

In order to finish the proof of Claim 1 it suffices to prove the validity of relation (\ref{Frobenius bar}). Since, by definition, the Frobenius conditions associated to the system $\Sigma$ are contained in the ideal $\mathcal{J}$, we have that
\[
\dfrac{\partial f_{ij}}{\partial \vec{y}}\cdot \left(
                                               \begin{array}{c}
                                                 f_{1k} \\
                                                 \vdots \\
                                                 f_{nk} \\
                                               \end{array}
                                             \right)=\dfrac{\partial f_{ik}}{\partial \vec{y}}\cdot \left(
                                               \begin{array}{c}
                                                 f_{1j} \\
                                                 \vdots \\
                                                 f_{nj} \\
                                               \end{array}
\right)\]
holds in a neighborhood of $Q$ relative to the variety $C$. Splitting the variables $\vec{y}$  into $\overline{\vec{y}}$ and $\widehat{\vec{y}}$ as before, this relation can be written
\[
\dfrac{\partial f_{ij}}{\partial \overline{\vec{y}}}\cdot \left( \begin{array}{c}
                                                 f_{1k} \\
                                                 \vdots \\
                                                 f_{rk} \\
                                               \end{array}\right)
+ \dfrac{\partial f_{ij}}{\partial \widehat {\vec{y}}}\cdot \left(
                                               \begin{array}{c}
                                                 f_{(r+1)k} \\
                                                 \vdots \\
                                                 f_{nk} \\
                                               \end{array}\right)
                                               =
\dfrac{\partial f_{ik}}{\partial \overline{\vec{y}}}\cdot
\left( \begin{array}{c}
                                                 f_{1j} \\
                                                 \vdots \\
                                                 f_{rj} \\
                                               \end{array}\right)
+ \dfrac{\partial f_{ik}}{\partial \widehat {\vec{y}}}\cdot \left(
                                               \begin{array}{c}
                                                 f_{(r+1)j} \\
                                                 \vdots \\
                                                 f_{nj} \\
                                               \end{array}\right)
\]
Identity (\ref{Frobenius bar}) is obtained from this formula simply by  replacing with relation (\ref{despeje}) for the indices $j$ and $k$. This finishes the proof of Claim 1.

\noi \emph{Proof of Claim 2.} Suppose that $\mu$ is a solution of $\overline{\Sigma}$ around $\overline{\vec{y}}_0$. Clearly the image of $\gamma:=(\mu, \varphi\circ\mu)$ is contained in $C$ because it is included in the the graph of $\varphi$. In particular $\gamma$ verifies the algebraic constraint $\vec{g}=0$.

\noi Decompose $\gamma=(\overline{\gamma},\widehat{\gamma})$ where $\overline{\gamma}:=\mu$ and $\widehat{\gamma}:=\varphi\circ\mu$. For all $j=1,\ldots ,m$ we have the relations
\[\dfrac{\partial \overline{\gamma}}{\partial x_j}=\dfrac{\partial \mu}{\partial x_j}=\left(
                                               \begin{array}{c}
                                                 f_{1j}(\mu, \varphi\circ\mu) \\
                                                 \vdots \\
                                                 f_{rj}(\mu, \varphi\circ\mu) \\
                                               \end{array}\right)
=\left(\begin{array}{c}
f_{1j}(\gamma) \\
                                                 \vdots \\
                                                 f_{rj}(\gamma) \\
                                               \end{array}\right)
\]
and
\[
\dfrac{\partial \widehat{\gamma}}{\partial x_j}=\dfrac{\partial (\varphi\circ \mu)}{\partial x_j}=\dfrac{\partial \varphi}{\partial \overline{\vec{y}}}\cdot \dfrac{\partial \mu}{\partial x_j}=\dfrac{\partial \varphi}{\partial \overline{\vec{y}}}\cdot \left(
                                               \begin{array}{c}
                                                 f_{1j}(\mu, \varphi\circ\mu) \\
                                                 \vdots \\
                                                 f_{rj}(\mu, \varphi\circ\mu) \\
                                               \end{array}\right)=
\left(
                                               \begin{array}{c}
                                                 f_{(r+1)j}(\gamma) \\
                                                 \vdots \\
                                                 f_{nj}(\gamma) \\
                                               \end{array}\right)
\]
where the last equality follows from identity (\ref{despeje}). This completes the proof of Claim 2 and therefore also of the lemma. \end{proof}\\

\subsection{An integrability criterion}\label{sec:integrability}

We introduce an increasing chain of radical polynomial ideals in $\C[\vec{y}]$ associated to a Pfaffian system as in (\ref{autsyst}) (or its dual counterpart: a descending chain of algebraic varieties in $\C^{m}$) which allows us to give a geometrical criterion for the solvability of the system.

\begin{definition} \label{chainideals}
Let $(\mathfrak{I}_p)_{p\in\N_0}$ be the sequence of radical polynomial ideals in  $\C[\vec{y}]$ defined recursively as follows:
\begin{itemize}
\item $\mathfrak{I}_0:=\sqrt{(\vec{g})}$, where the equations $\vec{g}=0$ define the algebraic constraint of $\Sigma$.
\item $\mathfrak{I}_1:=\sqrt{\mathfrak{F}+\widetilde{\mathfrak{I}_0}}$, where  $\mathfrak{F}\subset \C[\vec{y}]$ is the ideal generated by the Frobenius conditions as in Notation \ref{ideal}.
\item For $p\ge 1$, $\mathfrak{I}_{p+1}:=\sqrt{\widetilde{\mathfrak{I}_p}}$.
\end{itemize}
Here, for $p\ge 0$, $\widetilde{\mathfrak{I}_p}\subset\C[\vec{y}]$ denotes the prolongation ideal introduced in Definition \ref{rulodef}.

\noi By duality, we define for each $p\ge 0$ the set $\V_p\subset \C^n$ as the set of zeros of the ideal $\mathfrak{I}_p$ in the affine space $\C^{n}$.

We obtain an ascending chain of polynomial ideals $\mathfrak{I}_0\subseteq \cdots \subseteq \mathfrak{I}_p \subseteq \mathfrak{I}_{p+1} \subseteq \cdots$ and the corresponding descending chain of algebraic varieties $\V_0\supseteq\cdots \supseteq\V_p\supseteq \V_{p+1}\supseteq\cdots$.

Each step $\mathfrak{I}_p\rightsquigarrow\mathfrak{I}_{p+1}$ or $\V_p\rightsquigarrow \V_{p+1}$ is called a \emph{prolongation}.
\end{definition}

From Proposition \ref{rulo}, both chains $(\mathfrak{I}_p)_{p\in\N_0}$ and  $(\V_p)_{p\in\N_0}$ depend only on the Pfaffian system $\Sigma$ and, by Noetherianity, there exists a minimum integer $p_\infty$ where the prolongations become stationary. Moreover, from the definitions of the chains we observe that if an integer $p$ satisfies $\mathfrak{I}_p=\mathfrak{I}_{p+1}$ (resp. $\V_p=\V_{p+1}$) therefore the chain becomes stationary and hence $p\ge p_\infty$. In other words, $p_\infty$ is the smallest integer $p\ge 0$ such that $\mathfrak{I}_p=\mathfrak{I}_{p+1}$ (resp. $\V_p=\V_{p+1}$) holds. \\

This section is devoted to state integrability criteria for autonomous differential-algebraic Pfaffian systems  in terms of the prolongation varieties $\V_p$ introduced above.\\

As a consequence of Lemma \ref{key lemma auto} we deduce:

\begin{lemma} \label{key lemma p}
Let $\Sigma$ be an autonomous differential-algebraic Pfaffian system:
\[
\Sigma = \left\{
\begin{matrix}
\ \dfrac{\partial \vec{y}}{\partial \vec{x}}= \vec{f}(\vec{y}),\cr
\vec{g}(\vec{y})=0\
\end{matrix} \right.
\]
where $\vec{f}$ and $\vec{g}$ are finite sets of polynomials in $\C[\vec{y}]$. Let $p\in\N_0$ be such that $\V_p$ and $\V_{p+1}$ are the same nonempty variety in a neighborhood of a common regular point $Q\in\C^n$.
Then $\Sigma$ is integrable. Moreover, there exists a solution of $\Sigma$ passing through the point $Q$ and contained in $\V_{p+1}$.
\end{lemma}

\begin{proof} Let $\vec{g}_p\subset \C[\vec{y}]$ be a finite system of generators of $\mathfrak{I}_p$, the defining ideal of $\mathbb{V}_p$.  Consider the autonomous differential-algebraic Pfaffian system
\[
\Sigma_p = \left\{
\begin{matrix}
\ \dfrac{\partial \vec{y}}{\partial \vec{x}}= \vec{f}(\vec{y}),\cr
\vec{g}_p(\vec{y})=0\
\end{matrix} \right.
\]
Note that the system $\Sigma_p$ verifies the assumptions of Lemma \ref{key lemma auto} with $\mathcal{I} = \mathfrak{I}_p$ and $\mathcal{J}= \mathfrak{I}_{p+1}$. Then $\Sigma_p$ is integrable and there exists a solution of $\Sigma_p$ passing through $Q$ and contained in $\V_{p+1}$. The lemma follows from the fact that $(\vec{g}) \subset (\vec{g}_p)$.
\end{proof}

\begin{corollary} \label{lemma_fund_general}
If the system $\Sigma$ is not integrable then for all integer $p$ such that $\V_p\ne \emptyset$, the inequality $\dim \V_{p+1}<\dim\V_p$ holds.
\end{corollary}
\begin{proof}
 If the relation $\dim \V_p=\dim \V_{p+1}$ holds for some integer $p\ge 0$ with $\V_p\ne \emptyset$, since $\V_{p+1}\subseteq \V_p$, the varieties $\V_p$ and $\V_{p+1}$ share a common (nonempty) irreducible component $C$ of maximal dimension. Then, any regular point $Q\in C$ verifies the hypotheses of Lemma \ref{key lemma auto} and, therefore, $\Sigma$ is integrable.
\end{proof}\\

The following proposition, which is an immediate consequence of  Corollary \ref{lemma_fund_general}, gives a first criterion for the integrability of $\Sigma$:

\begin{proposition} \label{criterion1}
Let $\Sigma$ be an autonomous differential-algebraic Pfaffian system. Then, $\Sigma$ is integrable if and only if $\V_{p_\infty}\ne \emptyset$.
\end{proposition}

\begin{proof}
First observe that the condition $\V_{p_\infty}\ne \emptyset$ is clearly necessary because if $\gamma$ is a solution of $\Sigma$, its image is contained in any variety $\V_p$, in particular, in $\V_{p_\infty}$ and then, $\V_{p_\infty}\ne \emptyset$. Reciprocally, if $\V_{p_\infty}\ne \emptyset$ but $\Sigma $ is not integrable, Corollary \ref{lemma_fund_general} implies that the decreasing chain is not stationary at level $p_\infty$, which contradicts the definition of the integer $p_\infty$. The proposition follows. \end{proof}

\begin{remark}
In the language of Proposition \ref{criterion1}, the classical Frobenius Theorem can be stated as follows: assume that no algebraic constraints appear in the system $\Sigma$, then $\Sigma$ is \emph{completely} integrable if and only if $\V_{p_\infty}=\C^n$ (or equivalently $\mathfrak{I}_{p_\infty}=0$). Obviously in this case $p_\infty=0$.
\end{remark}

\begin{example} Consider the following autonomous Pfaffian system for $m=n=2$ with no algebraic constraints:
\[
\Sigma = \left\{ \begin{matrix}
\ \dfrac{\partial y_1}{\partial x_1}\ & = & \!\!\!\!\!\!\!\!\!\!\!\!\!\!\!\!\!\!y_1 &\qquad \dfrac{\partial y_1}{\partial x_2}\ & = & y_1^2 \cr \ &\ &\ &\ &\ & \cr
\dfrac{\partial y_2}{\partial x_1}\ & = & y_1y_2+1 &\qquad \dfrac{\partial y_2}{\partial x_2}\ & = & y_1^2  \end{matrix} \right. \qquad .
\]
In this example, $\mathfrak{I}_0 = 0$. The Frobenius conditions are $y_1^2$ and  $y_1^2(y_2+y_1-2)$; therefore, the polynomial ideal $\mathfrak{F}\subset \C[y_1,y_2]$ is the principal ideal generated by $y_1^2$ and hence its radical $\mathfrak{I}_1$ is generated by $y_1$. In order to obtain the ideal  $\mathfrak{I}_2$, we compute the polynomials $D_1(y_1)$ and $D_2(y_1)$ obtaining $y_1$ and $y_1^2$, respectively. Then $\mathfrak{I}_2=\sqrt{(y_1,y_1^2)}=(y_1)=\mathfrak{I}_1$ and so, $p_\infty=1$.  Since $(y_1)\ne \C[y_1,y_2]$ the system $\Sigma$ is \emph{integrable} (but not \emph{completely integrable}). Moreover, any solution $\gamma:=(\gamma_1,\gamma_2)$ must satisfy $\gamma_1(x_1,x_2)\equiv 0$, because its image must be included in the line $\{y_1=0\}$. Hence, $\gamma_2$ does not depend on $x_2$ and satisfies $\dfrac{\partial \gamma_2}{\partial x_1}=1$ and therefore   $\gamma_2(x_1)=x_1+\lambda$ for a suitable constant $\lambda\in \C$. In other words, the solutions of $\Sigma$ are $\gamma(x_1,x_2)=(0,x_1+\lambda)$.
\end{example}

Proposition \ref{criterion1} gives a conceptually simple criterion to decide the integrability of a Pfaffian system $\Sigma$: it suffices to compute $p_\infty$ and to check if the algebraic variety $\V_{p_\infty}$ is empty or not. Even if the integer $p_\infty$ exists by Noetherianity, \emph{a priori} it could be too big. However we will show that $p_\infty$ is in fact bounded by the dimension of the ambient space.

\begin{proposition} \label{component}
Let $\Sigma$ be an autonomous differential-algebraic Pfaffian system.  If,  for some $p_0\in \N_0$, $C\subseteq \V_{p_0}$ is an irreducible component which is also included in $\V_{p_0+1}$, then $C$ is an irreducible component of $\V_{p_\infty}$. Moreover, $\V_{p_\infty}=\V_{n+1}$ and then, $p_\infty\le n+1$.
\end{proposition}

\begin{proof}
Due to Lemma \ref{key lemma p},  we can guarantee that for any regular point $Q\in C$ passes locally the image of a solution of $\Sigma$. Since every solution of $\Sigma$ is contained in all the varieties $\V_p$, we conclude in particular that the closure of $\textrm{Reg}(C)$ is contained in $\V_{p_\infty}$. On the other hand, the regular points form a dense subset of $C$, so $C$ is an irreducible algebraic set contained in $\V_{p_\infty}$. Moreover, it is also an irreducible component of $\V_{p_\infty}$ because by assumption it is an irreducible component of $\V_p$ and $\V_{p_\infty}\subseteq \V_p$.

If $\V_{p}$ is empty for some index $p$, then $\V_{p_\infty}$ is empty too. Therefore, without loss of generality we may suppose $\V_{n+1}\ne \emptyset$. Let $C_{n+1}$ be an irreducible component of $\V_{n+1}$. There exists a decreasing chain of irreducible varieties $C_p$, $p=0,\ldots,n+1$, where each $C_p$ is an irreducible component of $\V_p$. If the sequence $(C_p)_{p}$ is strictly decreasing, we have $\dim(C_{n+1})< \dim(C_{n})<\cdots<\dim(C_1)<\dim(C_0)\le n$ and then $C_{n+1}=\emptyset$, leading to a contradiction. Hence the sequence $(C_p)_{p}$ stabilizes in $C_{n+1}$ for some index $p\le n$. Thus, $C_{n+1}$ is an irreducible component of $\V_{p_\infty}$. Since this holds for any irreducible component of $\V_{n+1}$, we deduce that $\V_{n+1}\subseteq \V_{p_\infty}$. The other inclusion is always true; therefore, we have $\V_{p_\infty}=\V_{n+1}$.
\end{proof}

\begin{remark} The variety  $M:=\emph{Reg}(\V_{p_\infty})\subseteq \C^{n}$ is the \emph{integral submanifold} associated to the system $\Sigma$ in the usual sense: \emph{for each point $Q\in M$, {its tangent space $T_Q M$ is spanned by the vector fields determined by the differential system
$\Sigma$}, and the dimension of $M$ is maximal with this property}. For instance, in the previous example we have   $\V_{p_\infty}=\{0\}\times \C$ and the solutions are $\gamma(x_1,x_2)=(0,x_1+\lambda)$.

Hence, the invariant
$\varrho:=\dim(\V_{p_\infty})$ is a measure of the integrability of the system $\Sigma$: it agrees with the maximal dimension of integral submanifolds of $\Sigma$. The extreme cases  $\varrho=n$, and $\varrho=-1$ correspond to the  complete integrability and the inconsistency of $\Sigma$, respectively. In the previous example we have $\varrho=1$.

\end{remark}

We summarize Propositions \ref{criterion1} and \ref{component} in the following criterion concerning the integrability of $\Sigma$:

\begin{theorem} \textbf{\emph{(Criterion for the integrability of a Pfaffian system)}} \label{criterio}
Let $\Sigma$ be an autonomous  differential-algebraic Pfaffian system. Then
\[{}\qquad \Sigma \ \textrm{is\ integrable\ } \Longleftrightarrow
\ \V_{p_\infty}\ne \emptyset \ \Longleftrightarrow \ \V_{n+1}\ne \emptyset. \qquad {\
\rule{0.5em}{0.5em}}\]
\end{theorem}

\bigskip

For a non-autonomous differential-algebraic Pfaffian system $\Sigma$, Theorem \ref{thm_non_aut} stated in the Introduction is proved by considering the associated autonomous system $\Sigma_{\rm aut}$ defined as in (\ref{sigma_aut}), and making all the constructions in this section from $\Sigma_{\rm aut}$.

\section{Quantitative Aspects}\label{sec:quantitative}

\subsection{Some tools from effective commutative algebra and algebraic geometry} \label{basic}

Throughout this section we will apply some results from effective commutative algebra and algebraic geometry. We recall them here in the precise formulations we will use.

\bigskip

One of the results we will apply is an effective version of the strong Hilbert's Nullstellensatz (see for instance \cite[Theorem 1.3]{Jelonek05}):

\begin{proposition}\label{prop:exprad}
Let $f_1,\dots, f_s\in {\C}[z_1,\dots, z_n]$ be polynomials of degrees bounded by $d$, and let $I=(f_1,\dots, f_s)\subset {\C}[z_1,\dots, z_n]$. Then $(\sqrt{I})\,^{d^n}\subset I$.
\end{proposition}

In order to obtain upper bounds for the number and degrees of generators of the radical of a polynomial ideal,  we will apply the following estimates, which follow from the algorithm
presented in \cite[Section 4]{Laplagne06} (see also \cite{KL91b}, \cite{KL91a}) and estimates for the number and degrees of polynomials involved in Gr\"obner basis computations (see, for instance, \cite{Dube90}, \cite{MM84} and \cite{Giusti84}):

\begin{proposition} \label{radicalgen}
 Let $I=(f_1,\dots, f_s)\subset {\C}[z_1,\dots, z_n]$ be an ideal  generated by $s$ polynomials of degrees at most $d$ that define an algebraic variety of dimension $r$ and let $\nu=\max\{1, r\}$. Then, the radical ideal $\sqrt{I}$ can be generated by $(sd)^{2^{O(\nu n)}}$ polynomials of degrees at most $(sd)^{2^{O(\nu n)}}$.
 \end{proposition}

\subsection{An effective decision method}

The integrability criteria presented in Section \ref{sec:integrability} enable the application of tools from effective algebraic geometry in order to derive a decision method for the Pfaffian systems under consideration.

As stated in Theorem \ref{criterio}, we have that for an autonomous differential system $\Sigma$ in $n$ differential unknowns, $\Sigma$ is integrable if and only if $\V_{n+1}\ne \emptyset$. We start by estimating the number and degrees of polynomials generating the intermediate ideals $\frak{I}_p$ and the complexity of computing them.

\begin{lemma}\label{gradogs} With our previous notation, let $\nu:=\max\{1, \dim(\V_0)\}$, $\sigma:=\max\{1, s\}$ and $d:=\max\{\deg(\vec{f}), \deg(\vec{g})\}$. There exists a universal constant $c>0$ such that for each $0\le p\le p_{\infty}$, the ideal $\mathfrak{I}_p$ can be generated by a family of polynomials $\vec{g}_p$ whose number and degrees are bounded by $(nm \sigma d)^{2^{c(p+1)\nu n}}$. These polynomials can be computed algorithmically within complexity $(n m\sigma d)^{2^{O(n^3)}}$.
\end{lemma}

\begin{proof} For $p=0$, we have that $\vec{g}_0$ is a set of generators of $\sqrt{(\vec{g})}$; then, by Proposition \ref{radicalgen}, they can be chosen to be at most $(\sigma d)^{2^{c_0\nu n}}$ polynomials of degrees bounded by $(\sigma d)^{2^{c_0\nu n}}$, where $c_0$ is a universal constant.

The ideal $\mathfrak{I}_1$ is the radical of the polynomial ideal generated by $\vec{g}_0$, $D_j(g)$ for every $g$ in $\vec{g}_0$ and $1\le j \le m$, and the Frobenius conditions $D_j(f_{ik})-D_k(f_{ij})$ for all indices $i,j,k$ with $i=1,\ldots,n$, and $j,k=1,\ldots,m$. For each polynomial $g$ in $\vec{g}_0$, we add $m$ polynomials $D_j(g)$, for $1\le j \le m$, and, since $\deg(\vec{f}) \le d$, we have that $\deg(D_j(g)) \le \deg(g) + d$. The Frobenius conditions are given by $n\binom{m}{2}$ polynomials of degrees bounded by $2d$. We may assume that the constant $c_0$ is chosen so that the number and degrees of all these polynomials is bounded by $(nm\sigma d)^{2^{c_0\nu n}}$.
Then, by Proposition \ref{radicalgen}, the radical ideal $\mathfrak{I}_1$ can be generated by a set of polynomials whose number and degrees are bounded by
$(nm\sigma d)^{2^{2c_0\nu n+1}} \le (nm\sigma d)^{2^{2c \nu n }}$ by taking, for instance, $c\ge c_0+1$.

Assume that, for $p\ge 1$, there is a system $\vec{g}_p$ of polynomials that generate $\mathfrak{I}_p$ whose number and degrees are bounded by $(nm \sigma d)^{2^{c(p+1)\nu n}}$. Recalling that  $\mathfrak{I}_{p+1}=\sqrt{\widetilde{\mathfrak{I}_p}} = \sqrt{(\vec{g}_p; D_j(\vec{g}_p), 1\le j \le m)}$, the bounds from Proposition \ref{radicalgen} imply that $\mathfrak{I}_{p+1}$ can be generated by a set of polynomials whose number and degrees are bounded by
\[\Big((m+1)((nm\sigma d)^{2^{c(p+1)\nu n}})^2\Big)^{2^{c_0 \nu n}} \le (nm\sigma d)^{2^{c(p+1)\nu n+c_0\nu n + 2}} \le (nm\sigma d)^{2^{c(p+2)\nu n}}\]
for a universal constant $c$ (it suffices to take $c\ge c_0+2$).

The complexity bound follows from the complexity of the computation of the radical of a polynomial ideal stated in \cite[Section 4]{Laplagne06}.
\end{proof}

\bigskip

As a consequence of the previous lemma, we deduce that we can obtain a Gr\"obner basis of the defining ideal of $\V_{n+1}$ within complexity $(nm\sigma d)^{2^{O(n^3)}}$, which enables us to decide immediately whether this variety is empty or not. We conclude:

\begin{theorem}\label{thm:decision}
Let $\Sigma$ be an autonomous differential-algebraic Pfaffian system:
\[
\Sigma = \left\{
\begin{matrix}
\ \dfrac{\partial \vec{y}}{\partial \vec{x}}= \vec{f}(\vec{y}),\cr
\vec{g}(\vec{y})=0\
\end{matrix} \right.
\]
where $\vec{y}= y_1,\dots, y_n$ and $\vec{x}= x_1,\dots, x_m$, $\vec{f}= (f_{ij})_{1\le i \le n, 1\le j \le m}$ and $\vec{g}=g_1,\dots, g_s$ polynomials in $\C[\vec{y}]$ with $\deg(\vec{f}), \deg(\vec{g}) \le d$. There is a deterministic algorithm that decides whether the system $\Sigma$ is integrable or not within complexity $(nm\sigma d)^{2^{O(n^3)}}$, where $\sigma = \max\{1, s\}$.
\end{theorem}

\subsection{An effective Differential Nullstellensatz for  differential-algebraic Pfaffian systems}

We write $\C\{\vec{y}\}$ for the polynomial ring in the infinitely many variables $y_{i,\alpha}$ where $\alpha\in \mathbb{N}_0^{m}$. This ring has naturally $m$ many independent derivations $\dfrac{\partial}{\partial x_1},\ldots,\dfrac{\partial}{\partial x_m}$: for $h\in \C\{ \vec{y} \}$,
\[
\dfrac{\partial h}{\partial x_j}:=\sum_{i,\alpha} \dfrac{\partial h}{\partial y_{i,\alpha}}\ y_{i,\alpha+e_j},
\]
where $e_j$ denotes the $j$-th vector of the canonical basis in $\mathbb{N}_0^m$. We identify the variable $y_i$ with  $y_{i,0}$ for $1\le i \le n$.

Consider an autonomous differential-algebraic Pfaffian system
\[
\Sigma = \left\{
\begin{matrix}
\ \dfrac{\partial \vec{y}}{\partial \vec{x}}= \vec{f}(\vec{y}),\cr
\vec{g}(\vec{y})=0\
\end{matrix} \right.
\]
where $\mathbf{f}$, $\mathbf{g}$ are polynomials in $\C[\vec{y}]$ of degrees bounded by an integer $d$. This sytem induces a differential ideal in $\C\{\vec{y}\}$ (that is, an ideal of $\C\{\vec{y}\}$ which is closed under the derivations $\dfrac{\partial}{\partial x_j}$) simply by taking the ideal generated by the polynomials $y_{i,e_j}-f_{ij},\ g_k$ (with $i=1,\ldots,n$, $j=1,\ldots,m$, $k=1,\ldots ,s$) and all their derivatives. We denote this ideal by $[\dfrac{\partial \vec{y}}{\partial \vec{x}}- \vec{f}, \vec{g}]$.

By the differential Nullstellensatz (see \cite{Kolchin}), $\Sigma$ has no solution if and only if $1$ lies in the differential ideal $[\dfrac{\partial \vec{y}}{\partial \vec{x}}- \vec{f}, \vec{g}]\subset \C\{ \vec{y}\}$.
The aim of this section is to prove an upper bound for the number of
derivations needed to obtain a representation of $1$ as an element
of this differential ideal assuming that $\Sigma$ is inconsistent.
A previous bound on this subject for general differential algebraic systems is given in \cite{Golubetal}.

\bigskip

For each $\alpha=(\alpha_1, \ldots ,\alpha_m) \in \N_0^m$, let $|\alpha|=\alpha_1+\ldots+\alpha_m$ and for each $k\ge 0$, let $\vec{y}^{[k]}$ be the set of all variables $y_{i, \alpha}$ with $|\alpha| \le k$. For $h\in\C\{\vec{y}\}$, we denote  $\dfrac{\partial^{|\alpha|}h}{\partial\vec{x}^\alpha} = \dfrac{\partial^{|\alpha|}h}{\partial x_1^{\alpha_1}\dots \partial x_m^{\alpha_m}}$ and $$h^{[k]}=\left\{\dfrac{\partial^{|\alpha|}h}{\partial\vec{x}^\alpha}: \ |\alpha|\le k\right\}.$$

We consider the algebraic ideals $\mathfrak{I}_p\subset \C[\vec{y}]$ $(0\le p \le p_\infty$) 
introduced in Notation \ref{chainideals}. For $p=0,\dots, p_\infty$, let $\mathbf{g}_p \subset \C[\vec{y}]$ be a system of generators of the ideal $\mathfrak{I}_p$. We also denote $\vec{g}_{-1}:= \vec{g}$. Since $(\vec{g}) \subset \mathfrak{I}_p$ for every $p$, the differential-algebraic Pfaffian system
\[
\Sigma_p = \left\{
\begin{matrix}
\ \dfrac{\partial \vec{y}}{\partial \vec{x}}= \vec{f}(\vec{y}),\cr
\vec{g}_p(\vec{y})=0\
\end{matrix} \right.
\]
has no solution and so, by the differential Nullstellensatz, $1\in[\dfrac{\partial \vec{y}}{\partial \vec{x}}- \vec{f}, \vec{g}_p]$.
Thus,  there exists a non negative
integer $k$ (depending on $p$) such that
$1\in((\dfrac{\partial \vec{y}}{\partial \vec{x}}- \vec{f})^{[k]}, \vec{g}_p^{[k]})\subseteq
\C[\vec{y}^{[k+1]}]$.
Moreover, by Proposition \ref{criterion1},  we have that $\{ \vec{g}_{p_\infty} = 0 \} = \mathbb{V}_{p_\infty} = \emptyset$, and so, $1 \in (\vec{g}_{p_\infty})\subset \C[\vec{y}]$.
We define
\begin{equation}\label{defk} k_p = \min \Big\{k\in\N_0 :  1\in\Big((\dfrac{\partial \vec{y}}{\partial \vec{x}}- \vec{f})^{[k]}, \vec{g}_p^{[k]}\Big)\Big\}.
\end{equation}
In particular,
\[k_{-1} = \min \Big\{k\in\N_0 :  1\in\Big((\dfrac{\partial \vec{y}}{\partial \vec{x}}- \vec{f})^{[k]}, \vec{g}^{[k]}\Big)\Big\}\]
is the order of differentiation of the input equations we want to bound.

Note that, since
$(\vec{g}_p)=\mathfrak{I}_p\subset \mathfrak{I}_{p+1}=(\vec{g}_{p+1})$ for every $p$, the sequence
$k_p$ is decreasing. In addition, as  $1 \in (\vec{g}_{p_\infty})$, we have that $k_{p_{\infty}}= 0$. We will obtain an upper bound for $k_{-1}$ by recursively computing upper bounds for $k_p$ for $p=p_{\infty},\dots,0$. In order to do this, we will use the following auxiliary sequence of non-negative integers defined for $p=0\dots, p_\infty$:
\[\varepsilon_0 := \min \{ \varepsilon \in \N : \mathfrak{I}_0^{\varepsilon} \subset (\vec{g})\},\]
\[\varepsilon_1 := \min \{ \varepsilon \in \N : \mathfrak{I}_1^{\varepsilon} \subset \widetilde{\mathfrak{I}}_{0}+ \mathfrak{F}\},\]
\[\varepsilon_p := \min \{ \varepsilon \in \N : \mathfrak{I}_p^{\varepsilon} \subset \widetilde{\mathfrak{I}}_{p-1}\}, \qquad \hbox{ for } 2\le p \le p_{\infty}.\]

By the definition of the ideals $\mathfrak{I}_p$ and $\widetilde{\mathfrak{I}}_p$, it follows that $\varepsilon_0$ is the Noether exponent of the ideal $(\vec{g})$, $\varepsilon_1$ is the Noether exponent of the ideal $\widetilde{\mathfrak{I}}_{0}+\mathfrak{F}$ and, for $p\ge 2$, $\varepsilon_p$ is the Noether exponent of $\widetilde{\mathfrak{I}}_{p-1}$. Taking into account that
$\mathfrak{F}\subset \Big((\dfrac{\partial \vec{y}}{\partial \vec{x}}- \vec{f})^{[1]}\Big) $ and, for $p\ge 1$, $\widetilde{\mathfrak{I}}_{p-1} = (\vec{g}_{p-1}; D_j(\vec{g}_{p-1}), 1\le j \le m)$ and $D_j(\vec{g}_{p-1}) \equiv \dfrac{\partial \vec{g}_{p-1}}{\partial x_j} \mod  \Big(\dfrac{\partial \vec{y}}{\partial \vec{x}}- \vec{f}\Big)$, it follows that
\begin{equation}\label{eq:eps}
\begin{split}
 \mathfrak{I}_0^{\varepsilon_0} \subset  \big( \vec{g}_{-1}\big), \quad
\mathfrak{I}_1^{\varepsilon_1} \subset  \Big((\dfrac{\partial \vec{y}}{\partial \vec{x}}- \vec{f})^{[1]}, \vec{g}_0^{[1]}\Big),  \\ \hbox{ and }\quad  \mathfrak{I}_p^{\varepsilon_p} \subset \Big(\dfrac{\partial \vec{y}}{\partial \vec{x}}- \vec{f},\vec{g}_{p-1}^{[1]}\Big)  \ \hbox{ for } p \ge 2.
\end{split}
 \end{equation}

The following is the key result that enables us to prove the required bound for the integers $k_p$:

\begin{lemma} \label{lem:derivrad} Let $0\le p\le p_{\infty}$ and $g \in \mathfrak{I}_p$. Then, for every $k\in \N_0$, all the partial derivatives $\dfrac{\partial^{|\alpha|} g}{\partial \vec{x}^\alpha}$, for $\alpha \in (\N_0)^m$ with $|\alpha| \le k$, lie in the polynomial ideal $\sqrt{\Big((\dfrac{\partial \vec{y}}{\partial \vec{x}}- \vec{f})^{[1+\varepsilon_p k]}, \vec{g}_{p-1}^{[1+\varepsilon_p k]}\Big)}$. 
\end{lemma}

\begin{proof} Consider the usual graded lexicographic order in $((\N_0)^m, \prec)$: for $\alpha, \beta \in (\N_0)^m$, $\alpha \prec \beta$ if and only if either $|\alpha|<|\beta|$ or $|\alpha|=|\beta|$ and there exists $j$, $1\le j \le m$, such that $\alpha_k = \beta_k$ for every $1\le k < j$ and $\alpha_j < \beta_j$. We proceed inductively following this order.

By the inclusions in (\ref{eq:eps}), we have that $g^{\varepsilon_p}\in \Big((\dfrac{\partial\vec{y}}{\partial \vec{x}}- \vec{f})^{[1]},\vec{g}_{p-1}^{[1]}\Big)$ and so, the statement holds for $k = 0$.

\bigskip
\noindent \textbf{Claim:} Given $\varepsilon \in \N$, for every $\alpha \in (\N_0)^m$, there exist $c\in \N$ and a differential polynomial $G\in \C\{\vec{y}\}$ in the ideal generated by the partial derivatives $\displaystyle{\left(\frac{\partial ^{|\beta|}g}{\partial \vec{x}^{\beta}}\right)}_{\!\!\beta\prec \alpha}$ such that
\begin{equation}\label{eq:derivatives}
\frac{\partial ^{\varepsilon|\alpha|}g^{\varepsilon}}{\partial \vec{x}^{\varepsilon\cdot\alpha}}=c\cdot \left(\frac{\partial ^{|\alpha|}g}{\partial \vec{x}^{\alpha}}\right)^\varepsilon + G.
\end{equation}

\noindent\emph{Proof of the claim.} According to the Leibniz formula, we have that
\[
\frac{\partial ^{\varepsilon |\alpha|}g^{\varepsilon}}{\partial \vec{x}^{\varepsilon \alpha}}=\sum_\nu c_\nu \left(\frac{\partial ^{|\nu_1|}g}{\partial \vec{x}^{\nu_1}}\right)\cdots \left(\frac{\partial ^{|\nu_\varepsilon|}g}{\partial \vec{x}^{\nu_\varepsilon}}\right) ,
\]
 where $c_\nu >0$ and $\nu$ runs over all matrices in $(\mathbb{N}_0)^{\varepsilon\times m}$, with rows $\nu_1,\ldots,\nu_\varepsilon\in (\mathbb{N}_0)^m$, such that the sums of their columns are  $\varepsilon\alpha_1,\ldots,\varepsilon\alpha_m$ respectively.

Let us analyze each term of the above sum. If, for some $i=1,\dots, \varepsilon$, we have that $|\nu_i|<|\alpha|$, the term lies in the ideal generated by $\displaystyle{\left(\frac{\partial ^{|\beta|}g}{\partial \vec{x}^{\beta}}\right)}_{\!\!\beta\prec \alpha}$ and it will contribute to the differential polynomial $G$.

Assume now that $|\nu_i| \ge |\alpha|$ for every $i$. Then, $|\nu_i| = \alpha$ for every $i$, since $|\nu_1|+\cdots +|\nu_\epsilon| = \varepsilon|\alpha|$.  If $\nu_i \prec \alpha$ for some $i$, the term will contribute to the polynomial $G$. Then, it remains to consider the case where no $\nu_i$ is smaller than $\alpha$ in the order $\prec$. We claim that in this case $\nu_i = \alpha$ for every $i$. Otherwise, let $i_0$ be such that $\nu_{i_0}$ is the smallest among all $\nu_i$ different from $\alpha$, and $j_0 = \min\{ j: \alpha_j < \nu_{i_0 j} \}$. By the minimality of $\nu_{i_0}$ and $j_0$, all the entries $\nu_{ij_0}$ of the $j_0$th column of the matrix $\nu$ are greater than or equal to $\alpha_{j_0}$ and for at least one index, $i_0$, the strict inequality holds; therefore, the sum of the $j_0$th  column of $\nu$ is greater than $\varepsilon\alpha_{j_0}$, which leads to a contradiction. Therefore, in the only term not containing a factor with $\nu_i\prec \alpha$ we have $\nu_i = \alpha$ for every $i$ and so, the term equals  $\displaystyle c\cdot \left(\frac{\partial ^{|\alpha|}g}{\partial \vec{x}^{\alpha}}\right)^\varepsilon $ for a combinatorial constant $c$.

\bigskip
Since $g^{\varepsilon_p}\in \Big((\dfrac{\partial \vec{y}}{\partial \vec{x}}- \vec{f})^{[1]},\vec{g}_{p-1}^{[1]}\Big)$, it follows that  $\dfrac{\partial^{\varepsilon_p|\alpha|} g^{\varepsilon_p}}{\partial \vec{x}^{\varepsilon_p\cdot\alpha}}\in \Big(\big(\dfrac{\partial \vec{y}}{\partial \vec{x}}- \vec{f}\big)^{[1+\varepsilon_p k]},\vec{g}_{p-1}^{[1+\varepsilon_p k]}\Big)$
for every $\alpha\in (\mathbb{N}_0)^m$ with $|\alpha|\le k$; in particular, $\dfrac{\partial^{\varepsilon_p|\alpha|} g^{\varepsilon_p}}{\partial \vec{x}^{\varepsilon_p\cdot\alpha}}$ lies in the radical of this polynomial ideal.
By induction in $((\N_0)^m, \prec)$, this property along with formula (\ref{eq:derivatives}) of the Claim imply that
\[\frac{\partial ^{|\alpha|}g}{\partial \vec{x}^{\alpha}}\in \sqrt{\Big(\big(\dfrac{\partial \vec{y}}{\partial \vec{x}}- \vec{f}\big)^{[1+\varepsilon_p k]},\vec{g}_{p-1}^{[1+\varepsilon_p k]}\Big)}\]
for every $\alpha\in(\mathbb{N}_0)^m$ with $|\alpha|\le k$.
\end{proof}

\begin{corollary} Let $0\le p\le p_{\infty}$. For every $k\in \mathbb{N}_0$, if $1\in\Big((\dfrac{\partial \vec{y}}{\partial \vec{x}}- \vec{f})^{[k]}, \vec{g}_p^{[k]}\Big)$, then $1\in\Big((\dfrac{\partial \vec{y}}{\partial \vec{x}}- \vec{f})^{[1+\varepsilon_p k]}, \vec{g}_{p-1}^{[1+\varepsilon_p k]}\Big)$. In particular, $k_{p-1} \le 1+\varepsilon_p k_p$.
\end{corollary}

\begin{proof}
 Since $\varepsilon_p\ge 1$, it follows that $(\dfrac{\partial \vec{y}}{\partial \vec{x}}- \vec{f})^{[k]} \subset (\dfrac{\partial \vec{y}}{\partial \vec{x}}- \vec{f})^{[1+\varepsilon_p k]}$.
 In addition, applying Lemma \ref{lem:derivrad} to the polynomials in $\vec{g}_p$, we deduce that $(\vec{g}_p^{[k]}) \subset \sqrt{\Big((\dfrac{\partial \vec{y}}{\partial \vec{x}}- \vec{f})^{[1+\varepsilon_p k]}, \vec{g}_{p-1}^{[1+\varepsilon_p k]}\Big)}$. Then,
 \[\Big((\dfrac{\partial \vec{y}}{\partial \vec{x}}- \vec{f})^{[k]}, \vec{g}_p^{[k]}\Big)\subset \sqrt{\Big((\dfrac{\partial \vec{y}}{\partial \vec{x}}- \vec{f})^{[1+\varepsilon_p k]}, \vec{g}_{p-1}^{[1+\varepsilon_p k]}\Big)},\]
 which implies the first assertion of the Corollary.

 The second assertion follows from the definition of the integers $k_p$ in equation (\ref{defk}).
\end{proof}

\begin{corollary}\label{cotaks} Let $\mu:=\min \{ 0\le
p\le p_\infty : k_p= 0\}$. Then,
${k_{-1}\le (\mu+1)\prod\limits_{p=0}^{\mu-1} \varepsilon_p}$.
\end{corollary}

\begin{proof} By applying recursively the inequality $k_{p-1}\le 1+\varepsilon_p k_p$ from the previous corollary, it follows easily that $k_{\mu-j} \le j \prod\limits_{i=\mu-j+1}^{\mu-1} \varepsilon_i$ for $j=1,\dots, \mu, \mu+1$. \end{proof}

\bigskip

Now, to prove our main result, it suffices only to bound the Noether exponents $\varepsilon_p$ for $p=0,\dots, \mu-1$, which will be easily done from Proposition \ref{prop:exprad}.

\begin{theorem} \label{diffnulls} Let $\vec{x}=x_1, \ldots, x_m$ independent variables and $\vec{y}=y_1, \ldots, y_m$ differential variables, and let  $\vec{f}=(f_{ij})_{1\le i\le n, 1\le j \le m} $ and $\vec{g}=g_1, \ldots, g_s $ be polynomials in $\C[\vec{y}]$ of degrees bounded by $d$. Let $\mathbb{V}\subset \C^{n}$ be the variety defined as the set of zeros of the ideal $(\vec{g})$ and $\nu:=\max\{1, \dim(\mathbb{V})\}$.
 Then,  \[1\in [\dfrac{\partial \vec{y}}{\partial \vec{x}}- \vec{f}, \vec{g}]\quad \Longleftrightarrow \quad 1\in \Big((\dfrac{\partial \vec{y}}{\partial \vec{x}}- \vec{f})^{[K]}, \vec{g}^{[K]}\Big),\]
where $K\le ((s+nm^2)d)^{2^{C \nu^2 n}}$ for a universal constant $C>0$.
\end{theorem}

\begin{proof} Following Corollaries \ref{cotaks} and \ref{lemma_fund_general}, notice first that $\mu\le p_\infty\le \nu+1$. On the other hand, by Proposition \ref{prop:exprad} and Lemma \ref{gradogs}, it follows that the Noether exponents $\varepsilon_p$ can be bounded as follows:
\[ \varepsilon_0 \le  d^n \qquad \hbox{and} \qquad \varepsilon_p \le (nm\sigma d)^{n 2^{c(p+1)\nu n}},\quad 1\le p\le p_{\infty}.\]
By Corollary \ref{cotaks}, we have that, for a suitable constant $c$,
\[ K = k_{-1} \le (n+2)\prod_{p=0}^\nu (nm\sigma d)^{n2^{c(p+1)\nu n}} \le  \]
\[ \le (n+2)(nm\sigma d)^{n2^{1+c(\nu+1)\nu n}} \le (nm\sigma d)^{2^{C \nu^2 n}} \]
for a suitable constant $C$.
\end{proof}

\bigskip

\noindent \textbf{Acknowledgments.} The authors thank Diego Rial (Universidad de Buenos Aires) for introducing them to the theory of Pfaffian systems.

\end{document}